\documentclass{article}
\usepackage{tikz}
\usepackage{amsmath,amsfonts,amssymb,graphicx}
\usepackage{mathrsfs}
\usepackage[latin1]{inputenc}
\usepackage{multicol}
\usepackage[all]{xy}
\usepackage{amsthm}
\usepackage{tcolorbox}
\usepackage{color}
\usepackage{cancel}
\usepackage[misc]{ifsym}

\newcommand{\dst}{\displaystyle}
\newcommand{\mb}{\mathbb}

\newcommand{\dist}{\text{dist}}

\newcommand{\Iso}{\text{Isom}}

\newcommand{\acts}{\curvearrowright}

\newtheorem{theorem}{Theorem}[section]
\newtheorem{proposition}[theorem]{Proposition}
\newtheorem{lemma}[theorem]{Lemma}
\newtheorem{corollary}[theorem]{Corollary}

\theoremstyle{definition}
\newtheorem{definition}[theorem]{Definition}

\theoremstyle{remark}
\newtheorem{example}[theorem]{Example}
\newtheorem{remark}[theorem]{Remark}

\begin{document}

\title{On the metric geometry of the space of compact balls and the shooting property for length spaces.}

\author{Waldemar Barrera,  Luis M. Montes de Oca and Didier A. Solis}

\date{\today}

\maketitle

\begin{abstract}
In this work we study the geodesic structure of the space $\Sigma (X)$ of compact balls of a complete and locally compact metric length space endowed with the Hausdorff distance $d_H$. In particular, we focus on a geometric condition (referred to as the shooting property) that enables us to give an explicit isometry  between $(\Sigma (X),d_H)$  and the closed half-space $ X\times \mb{R}_{\ge 0}$ endowed with a taxicab metric.

\end{abstract}



\section{Introduction}

In recent times, the interest for studying the geometry of the space of isometric classes of compact metric spaces $\mathfrak{M}$ have grown \cite{Memoli,Ivanov2,Ivanov}. Furnished with the Gromov-Hausdorff distance, the space $(\mathfrak{M},d_{GH})$ was shown to be geodesic \cite{Ivanov}. By the well known Gromov's Embedding Lemma, a family of compact metric spaces can be embedded in a $\ell^\infty$ space endowed with its Hausdorff distance $d_H$ \cite{Gromov}. Thus, in order to get some insight into the geodesic nature of $\mathfrak{M}$, in this work we analyze,the space of compact balls $(\Sigma (X),d_H)$. This sort of spaces have been studied extensively from the topological point of view in the context of hyperspace theory \cite{Nadler1,Nadler2}, but their geometric structure remains unexplored for the most part.

In particular, we focus on the geodesic structure of the space $(\Sigma (X),d_H)$ by means of a geometric condition relating geodesics and metric balls. Roughly speaking, the relevant geometric property states that any distance realizing curve going through the center of a metric ball keeps being distance realizing until it exits the ball; hence we denote this as the \emph{shooting property} of geodesics.

The present work is organized as follows: in section \ref{sec:prelim} we establish the notation and principal concepts to be used in this work.  In section \ref{sec:shooting} we introduce the shooting property and use it to give a description of the geodesic structure of the space $(\Sigma (X),d_H)$ in terms of a taxicab metric (Theorem \ref{main}) and provide explicit characterizations of $\Iso (\Sigma (\mathbb{R}^n),d_H)$ and $\Iso (\Sigma (\mathbb{H}^n(k)),d_H)$ (Corollary \ref{coro:isom}). We also show the set of points satisfying the shooting property is closed (Theorem \ref{ClosedShooting}). Finally, in section \ref{sec:apps} we further prove several different results related to this property. For instance, we show the stability of the shooting property under uniform convergence (Corollary \ref{corollary:uniform}) and study the behavior of the shooting property under quotients of isometric actions (Theorem \ref{theorem:quotient}).

%%%%%%%%%%%%%%%%%%%%%%
%%%%%%%%%%%%%%%%%%%%%%  Preliminares
%%%%%%%%%%%%%%%%%%%%%%

\section{Preliminaries}\label{sec:prelim}

Let us start by defining some basic concepts and establishing the notation to be used throughout this work. We will be using the notation and definitions found in the standard references \cite{Bridson,Burago,Plaut}.

We will be dealing meanly with geodesic (intrinsic) length spaces. Let us recall that a \emph{length space} is a metric space $(X,d)$ that satisfies $d=d_L$ where $d_L$ is the metric associated to the induced length structure
given by
\[
L(\gamma )=\sup_P\{S(P)\}.
\]
where $P$ denotes a partition $a=t_0<t_1<\cdots <t_{k-1}<t_k=b$ of the closed interval $[a,b]$ and
\[
S(P)=d(\gamma (a),\gamma (t_1))+d(\gamma (t_1),\gamma (t_2))+\cdots +d(\gamma (t_{k-1}),\gamma (b))
\]
denotes de length of the  corresponding polygonal curve approximating $\gamma$. In other words, we have the following:

\begin{definition} A metric space $(X,d)$ is a \emph{length space} if for all $p,q\in X$
\[
d(p,q)=\inf \{ L(\gamma) \mid \gamma \text{ is a curve joining }p\text{ and }q. \}
\]
If a curve $\gamma$ joining $p$ and $q$ is distance realizing, --that is, if $L(\gamma )=d(p,q)$-- we call it a \emph{geodesic segment}. If in a length space $(X,d)$ every pair of points can be joined by a geodesic segment we call such a space \emph{geodesic} or \emph{intrinsic}.
\end{definition}

Length spaces are generalizations of Riemannian manifolds and many celebrated results from Riemannian geometry do extend to this context, for instance,  the Hopf-Rinow theorem \cite{Burago}. In fact, it is possible to define a synthetic notion of curvature in a length space by comparing its geodesic triangles with the corresponding ones in to the standard two-dimensional complete Riemannian spaces of constant curvature. These categories of length spaces are called spaces of bounded curvature \cite{Plaut,Shiohama,Bridson}.

We now establish some notation and results pertaining the Hausdorff distance.

\begin{definition}\label{HausdorffDistance}
The \emph{Hausdorff distance} $d_H$ is defined by
    \begin{equation*}
    d_H(A,B) = \inf\{ r: A\subset U_r(B), B\subset U_r(A) \}
     \end{equation*}
where
$$U_r(A)=\{x\in X: \dist(x,A)<r\} = \dst\bigcup_{a\in A} B_r(a)$$
denotes the tubular neighborhood of $A$ of radius $r$ and $\dist(x,A)=\inf\{ d(x,a) : a\in A\}$.
\end{definition}

Sometimes we will use the following formula to compute the Hausdorff distance which is equivalent to Definition \ref{HausdorffDistance}.
\[
d_H(A,B) = \inf\{r: A\subset \overline{U}_{r}(B), B\subset \overline{U}_{r}(A) \},
\]
where $\overline{U}_{r}(A) =\{x\in X: \dist(x,A)\leq r\}$.

Also, we will often use an equivalent definition of the Hausdorff distance given in \cite{Burago} which adapts better to our context:
\begin{equation*}
d_H(A,B)=\dst\max\left\{ \dst\sup_{a\in A} \dist(a,B), \dst\sup_{b\in B} \dist(b,A) \right\}.
\end{equation*}
This latter formulation is easier to handle for certain computations. For instance, we can readily see that
the Hausdorff distance between two compact intervals $I=[a,b]$ and $J=[c,d]$ in $\mb{R}$ is given by
\[
d_H([a,b],[c,d]) = \max\{|c-a|,|d-b|\}.
\]

Notice as well that for subsets of an Euclidean space, the tubular neighborhoods satisfy the following important properties

\begin{proposition}\label{tubular}
Let $t,s>0$ and $A$ be a compact subset of $\mb{R}^{n}$. Then
\begin{enumerate}
\item[(a)] $\overline{U}_{t}(A)$ is compact.
\item[(b)] $\overline{U}_t(\overline{U}_s(A))=\overline{U}_{t+s}(A)$.
\end{enumerate}
\end{proposition}

Proposition \ref{tubular} is not a exclusive property of Euclidean spaces. In a complete and locally compact length space we have that every closed ball coincides with the closure of its open ball. Furthermore, we have that $\overline{U}_{t}(\overline{U}_{s}(A)) = \overline{U}_{t+s}(A)$ for every compact set $A$ and $t,s\geq 0$. Thus, in every complete and locally compact length space we can compute the Hausdorff distance between two closed balls by means of the following formula:
\[
d_{H}(\overline{B}_{t}(x), \overline{B}_{s}(y)) = \inf\{r: \overline{B}_{t}(x) \subset \overline{B}_{s+r}(y), \overline{B}_{s}(y) \subset \overline{B}_{t+r}(x)\}.
\]

Finally, given a metric space $(X,d)$, the space of  the compact balls of $X$ will be denoted by $(\Sigma (X),d_H)$.

%%%%%%%%%%%%%%%%%%%%%
%%%%%%%%%%%%%%%%%%%%%   Seccion sobre la propiedad del disparo
%%%%%%%%%%%%%%%%%%%%%

\section{An isometry between $(\Sigma(X),d_H)$ and $(X\times \mb{R}_{\geq 0}, d_T)$}\label{sec:shooting}

In this section we analyze the geodesic structure of the space of closed balls of a locally compact and complete length space $(X,d)$. First recall that for a locally compact and complete length space $(X, d)$ we have that all closed metric balls are compact and that the metric $d$ is intrinsic. Thus we can define the space $\Sigma (X)$ of compact balls of $(X,d)$ as
\[
\Sigma(X) = \{\overline{B}_{r}(x): p\in X, r\geq 0\},
\]

As means of motivation, we analyze the case $M=\mathbb{R}$. We readily notice that $\Sigma (\mathbb{R})$ can be parameterized in a rather simple way. Indeed,  any compact interval $I=[a,b]$ can be uniquely described in terms of its center $x\in\mb{R}$ and its radius $r\ge 0$ by
\[
x=\frac{a+b}{2},\qquad r=\frac{b-a}{2}.
\]
Thus $\Sigma (\mb{R})$ can be parameterized by the closed upper half plane
\[
\mb{R}\times\mb{R}_{\geq 0}= \{(x,r)\in\mb{R}^{2}: r\geq 0\}.
\]
Further, as we show below, the space $(\Sigma (\mb{R}),d_H)$ is \emph{isometric} to $\mb{R}\times\mb{R}_{\geq 0}$ endowed with the taxicab distance
\[
d_T((a,b),(x,y))= |a-x| + |b-y|.
\]

First notice that the function $f:(\Sigma (\mb{R}),d_H)\to (\mb{R}\times \mb{R}_{\geq 0}, d_T)$ given by
\[
f([x-r,x+r])=(x,r)
\]
 is clearly bijective. Now let $I=[x-s,x+s]$, $J=[y-t,y+t]$ be two compact intervals. Without loss of generality, assume $x\ge y$. Thus
\begin{eqnarray*}
d_H(I,J) &=&\max\{|x-s - (y-t)|, |x+s-(y+t)| \}\\
&=& \max \{|(x-y) + (t-s)|, |(x-y) - (t-s)|\}.
\end{eqnarray*}
Consider now the case $t\geq s$. Then $|(x-y) - (t-s)|\leq |(x-y) + (t-s)|$ and hence $$d_H(I,J) = |x-y| + |t-s|=d_T(f(I),f(J)).$$
Conversely, if $s\geq t$ then $|(x-y) - (s-t)|\leq |(x-y) + (s-t)|$. Thus
\begin{eqnarray*}
d_H(I,J) &=& \max \{|(x-y) - (s-t)|, |(x-y) + (s-t)|\} \\
&=& |x-y| + |t-s|\\
&=&d_T(f(I),f(J)).
\end{eqnarray*}
and hence $f$ is an isometry.

Based on the above example, the key idea in our approach consists in comparing $\Sigma (X)$ with the space $X\times\mb{R}_{\geq 0}$ endowed with the taxicab metric $d_T$ given by
\[
d_T((x,t), (y,s)) = d(x,y)+ |t-s|.
\]
In particular, we are interested in finding necessary and sufficient conditions for the map $f:X\times\mb{R}_{\geq 0}\to \Sigma(X)$ given by
\[
f(x,t)=\overline{B}_{t}(x).
\]
to be an isometry.

First notice that $f$ need not be even injective in general, as the following example illustrates:

\begin{example}
let us consider the metric space
\[
X=([-1,\infty)\times\{0\})\cup (\{0\}\times [0,1])\subset\mb{R}^2
\]
endowed with the metric length induced from the standard Euclidean metric of $\mb{R}^2$. Let $P=(-1,0)$ and $Q=(0,1)$, and notice then that $f(P,2)=\overline{B}_2(P)=\overline{B}_2(Q)=f(Q,2)$.
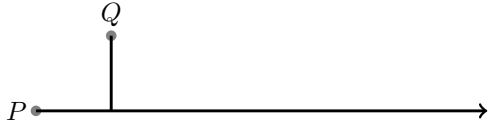
\begin{center}
  \begin{tikzpicture}[domain=0:3]
    \coordinate [label=left:$P$] (P) at (-1,0);
    \coordinate [label=above:$Q$] (Q) at (0,1);
    \foreach \point in {P,Q}
        \fill [black,opacity=.5] (\point) circle (2pt);
    \draw[line width=1.1pt,color=black,->] (-1,0) -- (5,0);
    \draw[line width=1.1pt,color=black] (0,0) -- (0,1);
  \end{tikzpicture}
\end{center}
\end{example}

As it turns out,  the lack of injectivity is one of the main obstruction for the map $f$ to be an isometry (see the proof of Theorem \ref{main} below). We begin by showing first that $f$ is a $1$-Lipschitz map.

%%%%%%%%%%%%%%%%%%%%%%%%%%%%%%%%%%%%
%%%%%%%%%%%%%%%%%%%%%%%%%%%%%%%%%%%%

\begin{proposition}
The function $f:(X\times\mb{R}_{\geq 0},d_T)\to (\Sigma(X),d_H)$, $f(x,t)=\overline{B}_{t}(x)$ is $1-$Lipschitz.
\end{proposition}

\begin{proof} By definition we have
\[
d_H(f(x,t),f(y,s))  = \max\left\{ \sup_{a\in \overline{B}_{t}(x)} \dist(a,\overline{B}_{s}(y)),
\sup_{b\in \overline{B}_{s}(y)} \dist(b,\overline{B}_{t}(x)) \right\}.
\]
Thus, for every $a\in\overline{B}_{t}(x)\setminus\overline{B}_{s}(y)$ we have $\dist(a,\overline{B}_{s}(y)) = d(a,y)-s\leq  d(a,x) + d(x,y)-s\leq t+d(x,y)-s$. Thus
\[
\sup_{a\in \overline{B}_{t}(x)} \dist(a,\overline{B}_{s}(y)) \leq d(x,y) + t-s.
\]
Proceeding in a similar way, we have
$$\sup_{b\in \overline{B}_{s}(y)} \dist(b,\overline{B}_{t}(x))\leq d(x,y)+s-t.$$
Thus
\[
d_H(f(x,t),f(y,s)) \leq \max\{d(x,y)+t-s,d(x,y)+s-t\} = d_T((x,t),(y,s)),
\]
and the proof is complete.
\end{proof}

In order to get a better insight, let us assume for the time being that $f$ is an isometry and consider the closed ball $\overline{B}_r(x)$. Let $y\neq x$ then
\[
\dist(y,\overline{B}_{r}(x)) \leq d(x,y) \leq \dst\sup_{a\in\overline{B}_{r}(x)} d(a,y),
\]
thus
\begin{eqnarray*}
d_H(f(x,r),f(y,0)) &=& d_H(\overline{B}_{r}(x),\{y\}) = \max\left\{ \dist(y,\overline{B}_{r}(x)), \dst\sup_{a\in\overline{B}_{r}(x)} d(a,y) \right\} \\ &=& \dst\sup_{a\in\overline{B}_{r}(x)} d(a,y).
\end{eqnarray*}
Since $f$ is an isometry we have $\dst\sup_{a\in\overline{B}_{r}(x)} d(a,y) = d(x,y) + r$. Furthermore, since  $\overline{B}_{r}(x)$ is a compact set, there exists $p\in \overline{B}_r(x)$ with
\[
d(p,y) = \dst\sup_{a\in\overline{B}_{r}(x)} d(a,y) = d(x,y) + r.
\]
Now we show that $p\in \partial \overline{B}_r(x)$. Indeed, notice that
\[
r + d(x,y) = d(p,y) \leq d(p,x) + d(x,y) \leq r + d(x,y),
\]
hence $d(p,x)=r$.

Finally, let us observe that $x,y,p$ are all collinear. Let  $\alpha:[0,d(x,y)]:\to X$ and $\beta:[0,r]\to X$
be geodesic segments -parameterized with respect to arc length- such that $\alpha(0)=y$, $\alpha(d(x,y))=x=\beta(0)$ and $\beta(r)=p$. Recall such paths exist since $(X,d)$ is geodesic. Consider then the path $\gamma=\alpha *\beta$, that is, $\gamma:[0,d(x,y)+r]\to X$ is given by
\[
\gamma(t) =
\left\{
\begin{array}{ll}
\alpha(t) & \mbox{if $t\in[0,d(x,y)]$} \\
\beta(t-d(x,y)) & \mbox{if $t\in[d(x,y), d(x,y)+r]$}
\end{array}
\right. .
\]
Let us show that for all $t,s\in[0,\ell+r]$ we have $d(\gamma(t),\gamma(s))=|t-s|$. If $t,s\in[0,d(x,y)]$ or $t,s\in[d(x,y), d(x,y)+r]$ then there is nothing to prove, since $\alpha$ and $\beta$ are geodesic segments parameterized by arc length. Thus let $t\in [0,d(x,y)]$ and $s\in[d(x,y),d(x,y)+r]$ and notice
\[
\begin{array}{rcl}
d(y,p) &=& d(\gamma(0),\gamma(\ell+r)) \\
&=& d(\gamma(0),\gamma(t)) + d(\gamma(t),\gamma(s)) + d(\gamma(s),\gamma(\ell+r)) \\
&\leq& d(\gamma(0),\gamma(t)) + d(\gamma(t),\gamma(\ell)) + d(\gamma(\ell),\gamma(s)) + d(\gamma(s),\gamma(\ell+r)) \\
&=& d(\alpha(0),\alpha(t)) + d(\alpha(t),\alpha(\ell)) + d(\beta(0),\beta(s-\ell)) + d(\beta(s-\ell),\beta(r)) \\
&=& \ell+r = d(y,x) + d(x,p) = d(y,p).
\end{array}
\]
It then follows that $d(\gamma(t),\gamma(s))=s-t=|t-s|$, which in turn implies that  $\gamma$ is a path of minimal length.

As it turns out, the existence of such a geodesic segment $\gamma$ for any choice of $x,y\in X$ and $r>0$ is also the sufficient condition we are looking for. Observe that in the case when $y\not\in \overline{B}_r (x)$ we can think of $\gamma$ as the path of a light ray shoot from $y$ and aimed to $x$. By connectedness, such ray has to enter $\overline{B}_r(x)$ at some point; whereas the above property shows that the ray has to leave $\overline{B}_r(x)$ after hitting $x$ as well. We capture this feature in the following definition.

\begin{definition}
Let $(X,d)$ be a length space. We say that $x$ satisfies the \emph{shooting property} if for any $r>0$ and any $y\neq x$ there exists a  point $p\in \partial\overline{B}_r(x)$ and a geodesic segment $\gamma :[0,d(x,y)+r]\to X$ such that $\gamma (0)=y$, $\gamma (d(x,y))=x$ and $\gamma (d(x,y)+r)=p$. We say that $(X,d)$ satisfies the shooting property if all of its points satisfy it.
\end{definition}

\begin{remark}
Notice that if $x$ satisfies de shooting property, then any geodesic through $x$ is distance realizing as long as it is defined and thus it is conjugate point free. This holds both for Riemannian manifolds \cite{chavel} and length spaces \cite{sormani}.
\end{remark}

\begin{example}
Let us notice that the shooting property does not hold even in the realm of the Riemannian model spaces. Even though the Euclidean spaces $\mb{R}^n$ and the hyperbolic spaces $\mb{H}^n(k)$ are readily seen to satisfy the shooting property, the spheres $\mb{S}^n(k)$ do not. To see this, consider $x\in \mb{S}^n(k)$ and choose $y$ to be the antipodal point to $x$, then every geodesic segment that joins $x$ and $y$ can not be extended as a distance realizing curve, therefore the shooting property does not hold on $x$.
\end{example}

\begin{example}\label{ex:noshoot}
Furthermore, there are some examples for which $f$ is injective, although the shooting property does not hold. For instance, consider the half space $\mathbb{R}\times \mathbb{R}_{\geq 0}$ endowed with the restriction of the Euclidian metric from $\mathbb{R}^{2}$, which will be denoted by $d|_{\mathbb{R}\times \mathbb{R}_{\geq 0}}$. Notice that no point $x\in (\mathbb{R}\times \mathbb{R}_{\geq 0}, d|_{\mathbb{R}\times \mathbb{R}_{\geq 0}})$  satisfies the shooting property. To see this, take a radius $r>0$ large enough such that $\overline{B}_{r}(x)$ intersects $\mathbb{R}\times \{0\}$ in an interval as the image below shows.

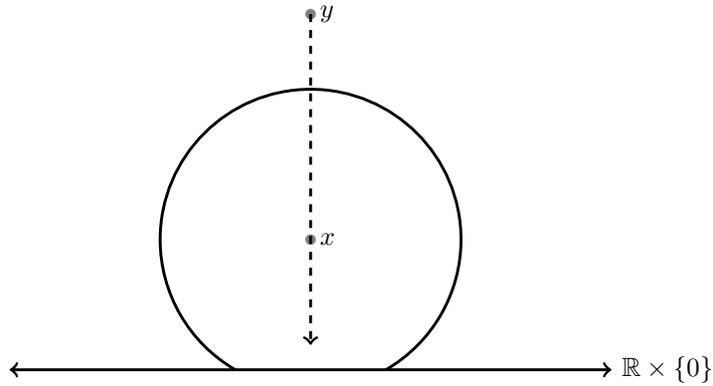
\begin{center}
\begin{tikzpicture}[domain=0:3]
\draw[line width=1.1pt,color=black,<->] (-4,-1.7320) -- (4,-1.7320) node[right] {$\mathbb{R}\times \{0\}$};
\draw[line width=1.1pt,color=black] (300:2cm) arc (300:360:2cm) arc (0:240:2cm);
\coordinate [label=right:$x$] (X) at (0,0);
\coordinate [label=right:$y$] (Y) at (0,3);
\foreach \point in {X,Y}
        \fill [black,opacity=.5] (\point) circle (2pt);
\draw[line width=1.1pt,color=black,dashed,->] (0,3) -- (0,-1.4);
\end{tikzpicture}
\end{center}

Further consider a point $y$ right above $x$, then any geodesic starting at $y$ and going through $x$ does not intersect $\partial \overline{B}_{r}(x)$, thus violating the shooting property.
\end{example}

The notion of the shooting property has to be defined pointwise. Indeed, there are locally compact and complete length spaces having some points in which the shooting property holds and others in which it doesn't hold, as the next example shows.

\begin{example}\label{Diamond}
 Let us consider the set
\[
X=\{(u,0)\in \mb{R}^{2}: |u|\geq \sqrt{2}\} \cup \{(u,v)\in \mb{R}^{2}: |u|+|v|=\sqrt{2}\},
\]
endowed with the intrinsic metric $d$ induced by the standard Euclidian metric in $\mb{R}^{2}$.
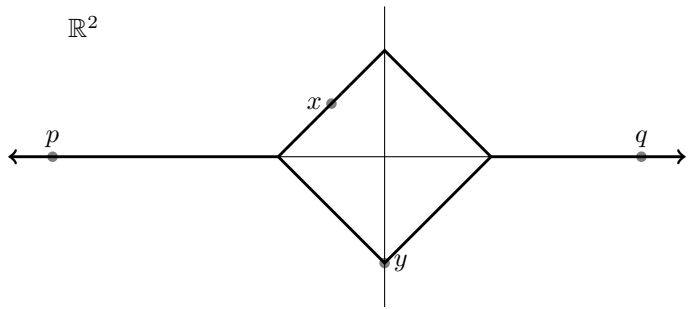
\begin{center}
\begin{tikzpicture}[domain=0:3]
    \draw (-5,0) -- (4,0);
    \draw (0,-2) -- (0,2);
    \coordinate (A) at (0,1.414);
    \coordinate (B) at (1.414,0);
    \coordinate [label=right:$y$] (C) at (0,-1.414);
    \coordinate (D) at (-1.414,0);
    \coordinate [label=below:$\mb{R}^{2}$] (R) at (-4,2);
    \coordinate [label=left:$x$] (X) at (-0.7071,0.7071);
    \coordinate [label=above:$p$] (P) at (-4.414,0);
    \coordinate [label=above:$q$] (Q) at (3.414,0);
    %\foreach \point in {A,B,C,D,X,P,Q}
    \foreach \point in {C,X,P,Q}
        \fill [black,opacity=.5] (\point) circle (2pt);
    \draw[line width=1.1pt,color=black,->] (1.414,0) -- (4,0);
    \draw[line width=1.1pt,color=black,->] (-1.414,0) -- (-5,0);
    \draw[line width=1.1pt,color=black] (A) -- (B) -- (C) -- (D) -- cycle;
  \end{tikzpicture}
\end{center}
We claim that the set of points satisfying the shooting property is
\[
Y=\{(u,0)\in\mb{R}^{2}: |u|\geq \sqrt{2} \}.
\]
In fact, is not difficult to see that any point in $Y$ satisfies the shooting property. On the other hand, set the points $x=(-\sqrt{2},\sqrt{2})$, $y=(0,-\sqrt{2})$, $p=(-4-\sqrt{2},0)$ and $q=(2+\sqrt{2},0)$. The ball $\overline{B}_{5}(x)$ has boundary $\partial \overline{B}_{5}(x) =\{p,q\}$ and we want to shoot from point $y$ in the direction of $x$. Observe that any path connecting $y$ with $p$ or $q$ and passing through $x$ has length at least 8, but $d(y,p)=6$ and $d(y,q)=4$. Therefore, all these paths can not be geodesic segments joining $y$ with the boundary of $\overline{B}_{5}(x)$. A similar idea proves that any point of the square ---except for $(\sqrt{2},0)$ and $(-\sqrt{2},0)$--- does not satisfy the shooting property, since we can take a large enough radius such that the boundary of the ball consists in two points.
\end{example}

A closer look at the above example reveals that no neighborhood of $(\sqrt{2},0)$ consists of points satisfying the shooting property and hence the set of points that satisfy the shooting property can not be open. Nonetheless, this set is always closed.

\begin{theorem}\label{ClosedShooting}
Let $(X,d)$ be a complete and locally compact length space. Suppose $\{x_n\}_{n=1}^{\infty}$ is a sequence satisfying the shooting property, that is,  $x_n$ satisfies the shooting property for every $n\in\mb{N}$. If $\{x_n\}_{n=1}^{\infty}$ converges to $x\in X$, then $x$ satisfies the shooting property.
\end{theorem}

\begin{proof}
Fix $y\neq x$ and $r>0$. Moreover, take $R>\max\{d(x,y),r\}$ large enough such that there exists $N\in\mb{N}$ such that $x_n\subset \overline{B}_{R}(x)$ for all $n\geq N$. We shall prove that there exists $p\in \partial \overline{B}_{r}(x)$ and a geodesic segment $\gamma:[0, d(x,y)+r]\to X$ such that $\gamma(0)=y$, $\gamma(d(x,y))=x$ and $\gamma(d(x,y)+r)=p$. Because $x_n$ satisfies the shooting property for every $n\in\mb{N}$ there exists $p_n\in\partial\overline{B}_{r}(x_n)$ and a geodesic segment $\gamma_{n}:[0,d(x_n,y)+r]\to X$ such that $\gamma_n(0)=y$, $\gamma_n(d(x_n,y))=x_n$ and $\gamma_n(d(x_n,y)+r)=p_n$. Moreover, since every geodesic segment $\gamma_n$ has length equals to $d(x_n,y)<R$, then applying Arzela-Ascoli Theorem (Theorem 2.5.14 of \cite{Burago}) in $\overline{B}_{R}(x)$ there exists a converging subsequence $\{\gamma_{n_i}\}_{i=1}^{\infty}$. Set $\gamma=\dst\lim_{i\to\infty} \gamma_{n_i}$ and $p=\dst\lim_{i\to \infty} p_{n_i}$. Using Proposition 2.5.17 of \cite{Burago} we get that $\gamma$ is a shortest path joining $y$ with $p$. Also, is not difficult to see that $p\in\overline{B}_{r}(x)$ and $\dst\lim_{i\to\infty} \gamma_{n_i}(d(x_i,y)) =x$ since $d(x_i,y)$ converges to $d(x,y)$ when $n\to\infty$. By a suitable reparameterization of $\gamma$ using the interval $[0,d(x,y)+r]$ we obtain the geodesic segment desired.
\end{proof}

%%%%%%%%%%%
%%%%%%%%%%%  Ultimos ejemplos
%%%%%%%%%%%

\begin{example}
Although the set of points satisfying the shooting property in a metric space is closed, it could be empty as Example \ref{ex:noshoot}  shows.
\end{example}

The next example exhibits a metric space where the set of points satisfying the shooting property is discrete and infinite.

\begin{example}
For every $n\in\mathbb{Z}$ define
\[
C_n = \{(a,b)\in \mathbb{R}^{2}: |x-n|+|y|=1\}.
\]
Now, let us consider the space
\[
C = \displaystyle\bigcup_{n\in\mathbb{Z}} C_{2n},
\]
endowed with the intrinsic metric induced by the restriction of the Euclidian metric of $\mathbb{R}^{2}$ on $C$.
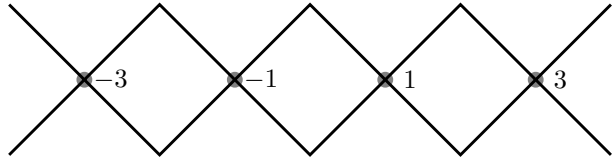
\begin{center}
\begin{tikzpicture}[domain=0:3]
\draw[line width=1.1pt,color=black] (-1,0) -- (0,1) -- (1,0) -- (0,-1) -- cycle;
\draw[line width=1.1pt,color=black] (1,0) -- (2,1) -- (3,0) -- (2,-1) -- cycle;
\draw[line width=1.1pt,color=black] (-3,0) -- (-2,1) -- (-1,0) -- (-2,-1) -- cycle;
\draw[line width=1.1pt,color=black] (4,1) -- (3,0) -- (4,-1);
\draw[line width=1.1pt,color=black] (-4,1) -- (-3,0) -- (-4,-1);
\coordinate [label=right:$-3$] (A) at (-3,0);
\coordinate [label=right:$-1$] (B) at (-1,0);
\coordinate [label=right:$\ 1$] (C) at (1,0);
\coordinate [label=right:$\ 3$] (D) at (3,0);
\foreach \point in {A,B,C,D}
        \fill [black,opacity=.5] (\point) circle (3pt);
\end{tikzpicture}
\end{center}
It is not difficult to check that any point of the set
\[
S=\{(2n+1,0)\in \mathbb{R}^{2}: n\in\mathbb{Z}\},
\]
satisfies the shooting property. Furthermore, by applying the same arguments as in  Example \ref{Diamond} we can show that no other point in $C$ satisfies the shooting property.
\end{example}

%%%%%%%%%%%%
%%%%%%%%%%%%  Fin de ejemplos
%%%%%%%%%%%%

We now establish the main result characterizing the spaces in which the shooting property holds.

\begin{theorem}\label{main}
Let $(X,d)$ be a locally compact and complete length space. Then the map $f:(X\times\mb{R}_{\geq 0},d_T)\to (\Sigma(X),d_H)$, $f(x,t)=\overline{B}_{t}(x)$ is an isometry if and only if $(X,d)$ satisfies the shooting property.
\end{theorem}

\begin{proof}
Let $x,y\in X$, $t,s\geq 0$, $\ell=d(x,y)$. We divide the proof in two parts:
\begin{enumerate}
\item[\textbf{(i)}] First,  let us assume that $d_H(\overline{B}_t(x), \overline{B}_{s}(y) ) = 0$. Thus
    \[
    \dst\sup_{a\in\overline{B}_{t}(x)} \dist(a,\overline{B}_{s}(y)) =0,
    \]
    that is, $\dist(a,\overline{B}_{s}(y))=0$ for all $a\in \overline{B}_{t}(x)$. Hence we have  $\overline{B}_{t}(x)\subset \overline{B}_{s}(y)$. We can show in an analogous way that $\overline{B}_{s}(y) \subset \overline{B}_{t}(x)$ and thus $\overline{B}_{s}(y) = \overline{B}_{t}(x)$. Since $(X,d)$ satisfies the shooting property, then there exist paths of minimal lenght $\alpha:[0,\ell+t]\to X$, $\beta:[0,\ell+s]\to X$, and points $p\in\partial\overline{B}_{t}(x)$, $q\in\partial \overline{B}_{s}(y)$ such that $\alpha(0)=y=\beta(\ell)$, $\alpha(\ell)=x=\beta(0)$, $p=\alpha(\ell+t)$ y $\beta(\ell+s)=q$. Moreover, since $p\in\partial\overline{B}_{t}(x)=\partial\overline{B}_{s}(y)$, then we have $d(y,p)=s$. Hence
    \[
    t + \ell = d(p,x) + d(x,y) = d(p,y) = s.
    \]
  By a similar argument we have $s+\ell = t$. It then follows that  $s=t$ and $\ell=0$, hence $(x,t)=(y,s)$. As a consequence we have that $f$ is injective.

\item[\textbf{(ii)}] Let us consider now $d_H(\overline{B}_t(x), \overline{B}_{s}(y) ) >0 $. Without loss of generality we can further assume
    \[
    \dst\sup_{a\in\overline{B}_{t}(x)} \dist(a,\overline{B}_{s}(y)) >0.
    \]
    Therefore we have
    \begin{eqnarray*}
    \sup\limits_{a\in\overline{B}_{t}(x)} \dist(a,\overline{B}_{s}(y)) &=&
    \sup \{ d(a,y) | {a\in\overline{B}_{t}(x)\setminus \overline{B}_{s}(y) }\} -s \\
    &\leq&
     \sup \{ d(a,x) + d(x,y) | {a\in\overline{B}_{t}(x)\setminus \overline{B}_{s}(y) } \} - s \\
    &\leq& d(x,y) + t -s.
    \end{eqnarray*}
   On the other hand, since $(X,d)$ satisfies the shooting property, there exist a geodesic segment $\gamma:[0,\ell+t]\to X$ and a point $p\in \partial\overline{B}_{t}(x)$ such that $\gamma(0)=y$, $\gamma(\ell)=x$ and $\gamma(\ell+t)=p$. Thus $d(p,y) = d(p,x) + d(x,y) = d(x,y) +t$ and hence
    \[
    \dst\sup_{a\in\overline{B}_{t}(x)} \dist(a,\overline{B}_{s}(y)) = d(x,y)  + t - s.
    \]
    Moreover, if
    \[
    \dst\sup_{b\in\overline{B}_{s}(y)} \dist(b,\overline{B}_{t}(x)) >0,
    \]
    then by a similar argument we end up with
    \[
    \dst\sup_{b\in\overline{B}_{s}(y)} \dist(b,\overline{B}_{t}(x)) = d(x,y)  +s-t
    \]
    and therefore
    \begin{eqnarray*}
    d_H(\overline{B}_{t}(x), \overline{B}_{s}(y)) &=&
    \max\left\{
    \dst\sup_{a\in\overline{B}_{t}(x)} \dist(a,\overline{B}_{s}(y)),
    \dst\sup_{b\in\overline{B}_{s}(y)} \dist(b,\overline{B}_{t}(x))
    \right\} \\
    &=&  \max\{ d(x,y) + t-s , d(x,y) + s-t\} \\ &=& d(x,y) + |t-s|.
    \end{eqnarray*}
    In the case
    \[
    \dst\sup_{b\in\overline{B}_{s}(y)} \dist(b,\overline{B}_{t}(x)) =0
    \]
    we have $\overline{B}_{s}(y) \subset \overline{B}_{t}(x)$, so $s\leq t$. Thus
    \begin{eqnarray*}
    d_H(\overline{B}_{t}(x), \overline{B}_{s}(y)) &=&
    \max\left\{
    \dst\sup_{a\in\overline{B}_{t}(x)} \dist(a,\overline{B}_{s}(y)), 0
    \right\} \\
    &=& d(x,y) + |t-s|.
    \end{eqnarray*}
Hence $f$ is an isometry.
\end{enumerate}
\end{proof}

We note that the shooting property is closely related to the existence of a correspondence between isometries of $(X,d)$ and isometries of $(\Sigma (X),d_H)$. Indeed, if $(X,d)$ satisfies the shooting property then any isometry of $f : (X,d)\to (X,d)$ gives rise to a natural isometry  $F :(\Sigma (X),d_H)\to (\Sigma (X),d_H)$ as follows: for any $B\in \Sigma (X)$, define $F(B)=f(B)$. Since $f$ is an isometry, it sends closed balls to closed balls and $F(\bar B_r(x))=\bar B_r(f(x))$. Then
 \begin{eqnarray*}
 d_H(F(\bar B_r(x)),F(\bar B_s(y))) &=&d_H(\bar B_r(f(x)),\bar B_s(f(y)))\\ &=&|r-s|+d(f(x),f(y))\\ &=& |r-s| +d(x,y)\\ &=& d_H(\bar B_r(x),\bar B_s(y)) .
 \end{eqnarray*}
Hence $F$ is an isometry of $(\Sigma (X),d_H)$.

On the other hand, additional properties can be used to show that ---in some specific cases--- in fact all isometries of $(\Sigma (X),d_H))$ arise in this way. For instance, in \cite{gruber1} it is shown that any isometry of $(\Sigma (\mathbb{R}^n),d_H)$ that sends one point sets to one point sets is a rigid motion. (See also \cite{gruber2}). In our setting, the relevant geometric property is based on the uniqueness of midpoints.

\begin{definition}
$(X,d)$ satisfies the \emph{tangency property} if for any pair of points $p,q\in X$ the intersection $\partial \bar B_{d(p,q)/2}(p)\cap \partial \bar B_{d(p,q)/2}(q)$ consists of a single point. In other words, the midpoint of the pair $(p,q)$ is unique.
\end{definition}

\begin{remark} \label{rem:uniquemp}
Let us notice that for $(X,d)$ satisfying the shooting property, $(\Sigma (X),d_H)$ does not satisfy the tangency property. In fact, $\bar B_{t}(z)$ is a midpoint  for the pair $(\bar B_r(p),\bar B_s(q))$, where $t=(r+s)/2$ and $z$ is any midpoint of $d(p,q)$. Thus if $(X,d)$ in addition satisfies the tangency property, then  $(\bar B_r(p),\bar B_s(q))$ has a unique midpoint in $(\Sigma (X), d_H)$ only when $r=s$.
\end{remark}

\begin{theorem}
 Let $(X,d)$ a locally compact and complete length space satisfying the shooting property. Then, every isometry of $(X,d)$ gives rise to an isometry of $(\Sigma (X),d_H)$. Further, if $(X,d)$ satisfies the tangency property then every isometry of $(\Sigma (X),d_H)$ gives rise to an isometry of $(X,d)$. 
 %and $\Iso (\Sigma (X),d_H)\cong \Iso (X,d)$.
 \end{theorem}

\begin{proof}  Let $F\in\Iso (\Sigma (X),d_H)$ and for any $t\ge 0$ let us denote $\Sigma_t=
\{\bar B_t(x)\ \mid \ x\in X \}\cong \{ t\}\times X$. We first proceed to show that $F$ maps balls of the same radius to balls of the same radius. In other words, given $t\ge 0$ there exists $r\ge 0$ such that $F(\Sigma_t )\subset \Sigma_r$. Thus consider $F(\bar B_t(p))=\bar B_r(\hat{p} )$, $F(\bar B_t(q))=\bar B_s(\hat{q} )$. Hence
\begin{equation*}
d(p,q)=d_H(\bar B_t(p),\bar B_t(q))=d_H(F(\bar B_t(p)),F(\bar B_t(q)))=|r-s|+d(\hat{p},\hat{q}).
\end{equation*}
Since $(\Sigma (X), d_H)$ is intrinsic, there exist midpoints between $F(\bar B_t(p))$ and $F(\bar B_t(q))$. Let $B\in \Sigma (X)$  such a midpoint. Thus
\begin{equation*}
d_H(F(\bar B_t(p)),B)=d_H(p,q)/2=d_H(F(\bar B_t(q)),B).
\end{equation*}
Further, since $F^{-1}\in\Iso \, (\Sigma (X),d_H)$ we have
\begin{equation*}
d_H(\bar B_t(p),F^{-1}(B))=d_H(p,q)/2=d_H(\bar B_t(q),F^{-1}(B))
\end{equation*}
Let $F^{-1}(B)=\bar B_u(a)$, thus
\begin{eqnarray*}
d_H(p,q)=d(p,q) &\le& d(p,a)+d(a,q)\\ &=& [d_H(\bar B_t(p),F^{-1}(B))-|t-u|]+[d_H(\bar B_t(q), F^{-1}(B))-|t-u|]\\ &=& d(p,q)-2|t-u| =d_H(p,q)-2|t-u|.
\end{eqnarray*}
It follows that $t=u$, so $B=F(\bar B_t(a))$. Thus %or equivalently, $B=F(a)=\bar B_t(\hat{a})$. 
\begin{eqnarray*}
d(p,a) &=& d_H(\bar B_t(p),\bar B_t(a))=d_H(F(\bar B_t(p)),B)=d(p,q)/2,\\
d(q,a) &=& d_H(\bar B_t(q),\bar B_t(a))=d_H(F(\bar B_t(q)),B)=d(p,q)/2,
\end{eqnarray*}
Since the tangent property holds in $(X,d)$ we then have that $a$ is the unique midpoint for $(p,q)$. As a consequence, the midpoint $B=F(\bar B_t(a))$ between $F(\bar B_t(p))$ and $F(\bar B_t(q))$ is also unique. By Remark \ref{rem:uniquemp} above we then have $r=s$ and hence $F(\Sigma_t)\subset \Sigma_r$. Further, given $B'\in \Sigma_r$ then $F^{-1}(B')\in \Sigma_t$. It follows that $F(\Sigma_t)=\Sigma_r$. Let $F(\{x\})=F(\bar B_0(x))=\bar B_R(\hat{x})$ and define $f:X\to X$ by $f(x)=\hat{x}$. We have just shown that $f$ is a distance preserving (hence injective) surjection. The proof is complete.
\end{proof}

In virtue of the above theorem, we can fully characterize the isometries of $(\Sigma (\mathbb{R}^n),d_H)$ and $(\Sigma (\mathbb{H}^n(k),d_H)$. %Indeed, since no Euclidean (hyperbolic) isometry maps $\mathbb{R}^n$ ($\mathbb{H}^n(k)$) into a proper subset of itself, the next result follows at once.

\begin{corollary}\label{coro:isom}
$\Iso(\Sigma (\mathbb{R}^n),d_H)\cong O(n)$ and $\Iso(\Sigma (\mathbb{H}^n),d_H)\cong O_0(n+1,1)$.
\end{corollary}

%%%\textcolor{red}{\textbf{Pendiente:} Ver detenidamente c\'omo se relaciona esto con el trabajo de Gruber. Pienso que si todo encaja bien, con esto podr\'{\i}amos rescatar algunas cuentas anteriores y completar un paper pequeno sobre la geometr\'{\i}a de $\mathcal{H}(\mathbb{R}^n)$ y $\mathcal{H}(\mathbb{H}^n)$}

\section{Some properties for length spaces with the shooting property}\label{sec:apps}

In metric geometry, as in many branches of mathematics, it is usual to consider standard constructions to build new objects out of simpler ones. In particular, isometric actions, products, quotients or limits of convergent sequences have been studied in different geometric contexts such as Riemannian manifolds or even for length spaces. In this section we want to analyze the behavior of the shooting property using these constructions.

We start by tackling the issue of stability of the shooting property under uniform convergence of length spaces. Let us recall that a sequence of metric spaces $(X_n, d_n)$ converges uniformly to a metric space $(X,d)$ if there exist metrics $\bar{d}_{n}$ such that every metric space $(X_n,d_n)$ is isometric to $(X,\bar{d}_n)$ and the sequence of metrics $\bar{d}_n$ converges uniformly to $d$. It means
\[
\lim_{n\to\infty} \dst\sup_{(a,b)\in X\times X} |\bar{d}_n(a,b)-d(a,b)| = 0
\]
Since every metric space $(X_n,d_n)$ is isometric to $(X,\bar{d}_n)$ we might even consider that $X_n=X$ and $\bar{d}_n= d_n$, just to simplify the notation.

We want to establish a relation between the Hausdorff distance in $X$ and the Hausdorff distance in $X_n$. In order to do this, let us denote by $d_H^{n}$ the Hausdorff distance in $(X_n,d_n)$ and the closed ball with center $x$ an radius $t$ by $\overline{B}_{t}^{n}(x)$.

\begin{lemma}\label{ballconvergence}
Let $(X_n,d_n)$ be a sequence of length spaces that converges uniformly to a length space $(X,d)$ and fix $x\in X$, $t\in \mb{R}_{\geq 0}$. Then for every $\varepsilon>0$ there is an $N\in\mb{N}$ such that
\[
\overline{B}_{t}(x) \subseteq \overline{B}_{t+\varepsilon}^{n}(x) \ \mbox{and} \ \overline{B}_{t}^{n}(x) \subseteq \overline{B}_{t+\varepsilon}(x),\ \forall n\ge N.
\]
\end{lemma}

\begin{proof}
Take $\varepsilon>0$. From  uniform convergence there exists $N\in\mb{N}$ such that
\[
-\varepsilon \leq d_n(a,b) - d(a,b) \leq \varepsilon,
\]
for all $a,b\in X$ and for all $n\ge N$. Thus for every $p\in \overline{B}_{t}(x)$ we have $d(p,x)\leq t$ and therefore
\[
d_n(p,x) \leq d(p,x) + \varepsilon \leq t + \varepsilon,
\]
then $d_n(p,x)\leq t+\varepsilon$, it means $p\in\overline{B}_{t+\varepsilon}^{n}(x)$. A similar procedure shows $\overline{B}_{t}^{n}(x) \subseteq \overline{B}_{t+\varepsilon}(x)$ and the proof is complete.
\end{proof}

\begin{proposition}\label{ballconvergencelimit}
Let $(X_n,d_n)$ be a sequence of length spaces that converges uniformly to a length space $(X,d)$. Then for every $x,y\in X$ and $t,s\in \mb{R}_{\geq 0}$ we have
\[
d_H(\overline{B}_{t}(x),\overline{B}_{s}(y)) =  \dst\lim_{n\to\infty} d_{H}^{n}(\overline{B}_{t}^{n}(x) , \overline{B}_{t}^{n}(y)),
\]
\end{proposition}

\begin{proof}
Let us recall that the Hausdorff distance between two compact sets $A$ and $B$ is given by
\[
d_H(A,B) = \inf\{r : A\subseteq \overline{U}_{r}(B), B\subseteq \overline{U}_{r}(A)\}.
\]
Moreover, because $(X,d)$ is a length space we have
\[
d_H(\overline{B}_{t}(x),\overline{B}_{s}(y)) = \inf\{r : \overline{B}_{t}(x)\subseteq \overline{B}_{s+r}(y), \overline{B}_{s}(y)\subseteq \overline{B}_{t+r}(x) \}.
\]
Take $\varepsilon>0$ and $\delta=\varepsilon /2$. We will proof that there exists $N\in\mb{N}$ such that
\[
|d_H(\overline{B}_{t}(x),\overline{B}_{s}(y)) - d_{H}^{n}(\overline{B}_{t}^n(x),\overline{B}_{s}^n(y))| \leq \varepsilon,
\]
for every $n\geq N$. Using Lemma \ref{ballconvergence} there exists $N\in \mb{N}$ such that for all $n\ge N$
\begin{equation}\label{eq:1}
\overline{B}_{t}(x) \subseteq \overline{B}_{t+\delta}^{n}(x) \ \mbox{and} \
\overline{B}_{t}^{n}(x) \subseteq \overline{B}_{t+\delta}(x),
\end{equation}
\begin{equation}\label{eq:2}
\overline{B}_{s}(y) \subseteq \overline{B}_{s+\delta}^{n}(y) \ \mbox{and} \
\overline{B}_{s}^{n}(y) \subseteq \overline{B}_{s+\delta}(y).
\end{equation}
Fix $n$ such that $n\geq N$. We divide the proof in two steps:
\begin{itemize}
\item[\textbf{(i)}.] Suppose $r$ satisfies $\overline{B}_{t}^{n}(x) \subseteq \overline{B}_{s+r}^{n}(y)$ and $\overline{B}_{s}^{n}(y) \subseteq \overline{B}_{t+r}^{n}(x)$. Using eq. \ref{eq:1} we get
    \[
    \overline{B}_{t}(x) \subseteq \overline{B}_{t+\delta}^{n}(x) \subseteq \overline{B}_{s+r+\delta}^{n}(y),
    \]
    therefore $\overline{B}_{t}(x)\subseteq \overline{B}_{s+r+\delta}^{n}(y)$. Now, using eq. \ref{eq:2} we have
    \[
    \overline{B}_{s+r+\delta}^{n}(y) \subseteq \overline{B}_{s+r+\delta+\delta}(y) = \overline{B}_{s+r+\varepsilon}(y),
    \]
    thus $\overline{B}_{t}(x) \subseteq \overline{B}_{s+r+\varepsilon}(y)$. A similar calculation leads us to $    \overline{B}_{s}(y) \subseteq \overline{B}_{t+r+\varepsilon}(x)$.
    These last contentions imply
    \[
    d_H(\overline{B}_{t}(x),\overline{B}_{s}(y)) \leq r + \varepsilon,
    \]
    so taking the infimum over all the numbers $r$ satisfying our assumptions we get
    \[
    d_H(\overline{B}_{t}(x),\overline{B}_{s}(y)) \leq d_{H}^{n}(\overline{B}_{t}^n(x),\overline{B}_{s}^n(y)) +\varepsilon,
    \]
    which implies
    $$d_H(\overline{B}_{t}(x),\overline{B}_{s}(y)) - d_{H}^{n}(\overline{B}_{t}^n(x),\overline{B}_{s}^n(y)) \leq \varepsilon .$$
\item[\textbf{(ii)}.] Using the same ideas from step (i) we get
    \[
    d_{H}^{n}(\overline{B}_{t}^n(x),\overline{B}_{s}^n(y)) - d_H(\overline{B}_{t}(x),\overline{B}_{s}(y)) \leq \varepsilon,
    \]
    which is equivalent to $-\varepsilon \leq d_H(\overline{B}_{t}(x),\overline{B}_{s}(y)) - d_{H}^{n}(\overline{B}_{t}^n(x),\overline{B}_{s}^n(y))$.
\end{itemize}
In conclusion
\[
|d_H(\overline{B}_{t}(x),\overline{B}_{s}(y)) - d_{H}^{n}(\overline{B}_{t}^n(x),\overline{B}_{s}^n(y))| \leq \varepsilon,
\]
as we desired.
\end{proof}

As a corollary from Proposition \ref{ballconvergencelimit} we can now establish the stability of the shooting property under uniform limits.

\begin{corollary}\label{corollary:uniform}
Let $(X_n,d_n)$ a sequence of length spaces that converges uniformly to a length space $(X,d)$, and suppose that every metric space $(X_n,d_n)$ satisfies the shooting property. Then $(X,d)$ satisfies the shooting property.
\end{corollary}

\begin{proof}
Just observe that
\[
\begin{array}{rcl}
d_H(\overline{B}_{t}(x),\overline{B}_{s}(y)) &=& \dst\lim_{n\to\infty} d_H^{n}(\overline{B}_{t}^{n}(x),\overline{B}_{s}^{n}(y)) \\
&=& \dst\lim_{n\to\infty} \left( d_n(x,y) + |t-s| \right) \\
&=& d(x,y) + |t-s|,
\end{array}
\]
thus $(X,d)$ satisfies the shooting property as we desired.
\end{proof}

We now move our attention to the study the cartesian product of spaces satisfying the shooting property.

\begin{proposition}
Let $(X,d_X)$ and $(Y,d_Y)$ be complete and locally compact length spaces. Suppose that $x$ and $y$ satisfy the shooting property in $X$ and $Y$, respectively. Then $(x,y)$ satisfies the shooting property in $(X\times Y,d_{X\times Y})$.
\end{proposition}

\begin{proof}
Take $(a,b)\in X\times Y$ and $r>0$. Since $x$ and $y$ satisfy the shooting property, according to the definitions there exist two points $p$ and $q$ such that $p\in\partial\overline{B}_{{r}/{\sqrt{2}}}(x)$ and $q\in\partial\overline{B}_{{r}/{\sqrt{2}}}(y)$ and  geodesic segments $\alpha$ (joining $a$ with $p$) and $\beta$ (joining $b$ with $q$), passing through $x$ and $y$, respectivley. Applying Lemma 3.6.4 of \cite{Burago}, the product of two geodesic segments is a geodesic segment in $X\times Y$ and therefore we have a geodesic segment joining $(a,b)$ with $(p,q)$ passing through $(x,y)$. Finally observe that
\[
d_{X\times Y}((x,y),(p,q)) = \dst\sqrt{ d_X(x,p)^2 + d_Y(y,q)^2} = \dst\sqrt{ \left(r/\sqrt{2}\right)^2 + \left(r/\sqrt{2}\right)^2} =r,
\]
thus $(p,q)\in \partial\overline{B}_{r}(x,y)$. This completes the proof.
\end{proof}

\begin{corollary}
Let $(X,d_X)$ and $(Y,d_Y)$ be complete and locally compact length spaces satisfying the shooting property. Thus so does $(X\times Y,d_{X\times Y})$.
\end{corollary}

Besides the standard product metric $d_{X\times Y}$, there are other choices of metrics in the cartesian product $X\times Y$ that are useful in certain geometric applications. For instance, the following metric has been studied in the context of hyperbolic complex manifolds \cite{Kobayashi}.

For any two metric spaces $(X,d_X)$ and $(Y,d_Y)$ it is possible to define a metric on $X\times Y$ by
\[
d_{\infty}((x,y),(a,b)) = \max\{d_X(x,a), d_Y(y,b)\}.
\]
We call $d_\infty$ the \emph{maximum metric} in $X\times Y$. If $X$ and $Y$ are complete and locally compact metric spaces then $(X\times Y,d_{\infty})$ is also complete and locally compact. Moreover, if $X$ and $Y$ are length spaces, then $(X\times Y,d_{\infty})$ is a length space. In fact, for any two points $(x,y),(a,b)\in X\times Y$, the point $(z,c)$ (where $z$ and $c$ are midpoints of $x,y$ and $a,b$ respectively) is a midpoint for $(x,y)$ and $(a,b)$. Since the space is complete we conclude that $X\times Y$ is a length space.

As an application of our previous results we show that the cartesian product of two length spaces satisfying the shooting property also has this property relative to the maximum metric.

\begin{lemma}\label{ballinfty}
For every $(x,y)\in X\times Y$ and $r\geq 0$ we have
\[
\overline{B}_{r}(x,y) = \overline{B}_{r}(x)\times \overline{B}_{r}(y).
\]
\end{lemma}

We are interested in finding the Hausdorff distance between two closed balls in $(X\times Y,d_{\infty})$.

\begin{proposition}\label{HausdorffBallsInftyMetric}
For every $(x,y),(a,b)\in X\times Y$ and $t,s\in\mb{R}_{\geq 0}$ we have
\[
d_{H}^{\infty}\left( \overline{B}_{t}(x,y), \overline{B}_{s}(a,b)  \right) = \max\left\{ d_H^{X}(\overline{B}_{t}(x),\overline{B}_{s}(a)), d_H^{Y}(\overline{B}_{t}(y),\overline{B}_{s}(b))  \right\},
\]
where $d_{H}^{\infty}$, $d_{H}^{X}$ and $d_{H}^{Y}$ are the Hausdorff distances in $(X\times Y,d_{\infty})$, $X$ and $Y$, respectively.
\end{proposition}

\begin{proof}
As in the proof of Proposition \ref{ballconvergencelimit} we will use that
\[
d_{H}^{\infty}\left( \overline{B}_{t}(x,y), \overline{B}_{s}(a,b)  \right) = \inf\left\{ r: \overline{B}_{t}(x,y)\subset \overline{B}_{s+r}(a,b),
\overline{B}_{s}(a,b)\subset \overline{B}_{t+r}(x,y)   \right\} .
\]
The proof is divided in two parts:
\begin{enumerate}
\item[\textbf{(i)}] Suppose $r$ satisfies
\begin{equation}\label{eq:3}
\overline{B}_{t}(x,y)\subseteq \overline{B}_{s+r}(a,b) \ \mbox{and} \
\overline{B}_{s}(a,b)\subseteq \overline{B}_{t+r}(x,y).
\end{equation}
Using Lemma \ref{ballinfty} we get
\begin{equation}\label{eq:4}
\overline{B}_{t}(x)\times \overline{B}_{t}(y) =  \overline{B}_{t}(x,y)\subseteq \overline{B}_{s+r}(a,b) = \overline{B}_{s+r}(a) \times \overline{B}_{s+r}(b),
\end{equation}
then $\overline{B}_{t}(x)\times \overline{B}_{t}(y)\subseteq \overline{B}_{s+r}(a) \times \overline{B}_{s+r}(b)$ and thus
\begin{equation}\label{eq:5}
\overline{B}_{t}(x) \subseteq \overline{B}_{s+r}(a) \ \mbox{and} \
\overline{B}_{t}(y)\subseteq \overline{B}_{s+r}(b).
\end{equation}
A similar procedure shows that
\begin{equation}\label{eq:6}
\overline{B}_{s}(a) \subseteq \overline{B}_{t+r}(x) \ \mbox{and} \ \overline{B}_{s}(b) \subset \overline{B}_{t+r}(y).
\end{equation}
Thus eqs. \ref{eq:3}, \ref{eq:4}, \ref{eq:5} and \ref{eq:6} yield
\[
d_{H}^{X}(\overline{B}_{t}(x),\overline{B}_{s}(a)) \leq r \ \mbox{and} \ d_{H}^{Y}(\overline{B}_{t}(y),\overline{B}_{s}(b)) \leq r,
\]
which implies
\[
\max\left\{ d_{H}^{X}(\overline{B}_{t}(x),\overline{B}_{s}(a)),  d_{H}^{Y}(\overline{B}_{t}(y),\overline{B}_{s}(b)) \right\} \leq r.
\]
Because of the way we choose $r$ we conclude
\[
\max\left\{ d_{H}^{X}(\overline{B}_{t}(x),\overline{B}_{s}(a)),  d_{H}^{Y}(\overline{B}_{t}(y),\overline{B}_{s}(b)) \right\} \leq d_{H}^{\infty}\left( \overline{B}_{t}(x,y), \overline{B}_{s}(x,y) \right).
\]
\item[\textbf{(ii)}] Suppose $r_1$ and $r_2$ satisfy $\overline{B}_{t}(x) \subseteq \overline{B}_{s+r_1}(a)$, $\overline{B}_{s}(a) \subseteq \overline{B}_{t+r_1}(x)$, $\overline{B}_{t}(y) \subseteq \overline{B}_{s+r_2}(b)$ and $\overline{B}_{s}(b) \subseteq \overline{B}_{t+r_2}(y)$. Take $r=\max\{r_1,r_2\}$ and using again Lemma \ref{ballinfty} we have
    \[
    \begin{array}{rcl}
    \overline{B}_{t}(x,y) &=& \overline{B}_{t}(x)\times \overline{B}_{t}(y) \\
    &\subseteq & \overline{B}_{s+r_1}(a) \times \overline{B}_{s+r_2}(b) \\
    &\subseteq & \overline{B}_{s+r}(a) \times \overline{B}_{s+r}(b) \\
    &=& \overline{B}_{s+r}(a,b),
    \end{array}
    \]
    therefore $\overline{B}_{t}(x,y) \subseteq \overline{B}_{s+r}(a,b)$. A similar calculation proves that $\overline{B}_{s}(a,b) \subseteq \overline{B}_{t+r}(x,y)$. Thus
    \[
    d_{H}^{\infty}\left( \overline{B}_{t}(x,y), \overline{B}_{s}(a,b) \right) \leq r.
    \]
    Because of the way we take $r$ we conclude
    \[
    d_{H}^{\infty}\left( \overline{B}_{t}(x,y), \overline{B}_{s}(a,b) \right) \leq \max\left\{ d_H^{X}(\overline{B}_{t}(x),\overline{B}_{s}(a)), d_H^{Y}(\overline{B}_{t}(y),\overline{B}_{s}(b))  \right\}.
    \]
\end{enumerate}
Now the proof is complete.
\end{proof}

As a consequence of Proposition \ref{HausdorffBallsInftyMetric} we have the following corollary.

\begin{corollary}
Let $(X,d_X)$ and $(Y,d_Y)$ be complete and locally compact length spaces satisfying the shooting property, then $(X\times Y, d_{\infty})$ satisfies the shooting property.
\end{corollary}

\begin{proof}
Just note that
\[
\begin{array}{rcl}
d_{H}^{\infty}\left( \overline{B}_{t}(x,y), \overline{B}_{s}(a,b)  \right) &=& \max\left\{ d_H^{X}(\overline{B}_{t}(x),\overline{B}_{s}(a)), d_H^{Y}(\overline{B}_{t}(y),\overline{B}_{s}(b))  \right\} \\
&=& \max\{ d_X(x,a) + |t-s|, d_Y(y,b) + |t-s| \} \\
&=& \max\{d_X(x,a),d_Y(y,b)\} + |t-s| \\
&=& d_{\infty}((x,y),(a,b)) + |t-s|,
\end{array}
\]
and the proof is finished.
\end{proof}

The relation between isometric actions of compact Lie groups on Riemannian manifolds has been studied intensively because these interactions provide important information of its geometric structure. For example, Myers and Steenrod proved in \cite{MS} that the group of isometries of a Riemannian manifold is a Lie group, and the same statement remains true for length spaces which are locally compact, with finite Hausdorff dimension and curvature bounded by below \cite{Fukaya-Yamaguchi}. Moreover, if an Alexandrov space have an isometry group of maximal size, then it is isometric to a Riemannian manifold \cite{GalazGuijarro}. An other example appears in \cite{Berestovskii}, where Berestovskii proves that a finite dimensional homogeneous metric space with curvature bounded by below is a smooth manifold.

Let us denote by $G\acts X$ an isometric action of $G<\Iso(X,d)$ on $X$ and  denote by $(\mathcal{H}(X), d_H)$ de space of all compact subsets of $X$ endowed with the Hausdorff distance. For every $g\in G$ we define $\mathcal{H}[g]: \mathcal{H}(X)\to \mathcal{H}(X)$ as $\mathcal{H}[g](K)=g(K)$. It is not hard to prove that $\mathcal{H}[g]$ is an isometry of $(\mathcal{H}(X),d_H)$ and $\mathcal{H}[f\circ g^{-1}] = \mathcal{H}[f]\circ \mathcal{H}[g]^{-1}$ for every $f,g\in G$. This implies that
\[
\mathcal{H}[G] = \{ \mathcal{H}[g]: g\in G \}
\]
is a subgroup of $\Iso(\mathcal{H}(X),d_H)$. In particular $\mathcal{H}[G]$ acts by isometries on $\Sigma(X)$.

Let us recall that isometric actions do not always give rise to metric quotients. However, if the isometric action  $G\acts X$ is \emph{proper} then ${X}/{G}$ es metric space.

 According to \cite{Bridson}, an isometric action $G\acts X$ is proper if for each $x\in X$ there exists $r>0$ such that the set
\[
\{g\in G: g(B_r(x))\cap B_{r}(x)\}
\]
is finite.

As the next result shows, the shooting property enables us to transfer properness from $(X,d)$ to $(\Sigma (X),d_H)$.

\begin{proposition}\label{propershooting}
Let $(X,d)$ be a length space satisfying the shooting property and $G\acts X$ a proper isometric action. Then the isometric action from $\mathcal{H}[G]$ on $\Sigma(X)$ is proper.
\end{proposition}

\begin{proof}
Take $\overline{B}_r(x)\in \Sigma(X)$ and let us denote the closed ball with center $K$ and radius $R$ in $(\Sigma(X),d_H)$ by $\overline{B}^{H}_{R}(K)$. Because the action $G\acts X$ is proper, there exists $R>0$ such that
\[
H=\{g\in G: g(B_{R}(x))\cap B_{R}(x)\neq \varnothing\}
\]
is a finite set, and suppose that $\mathcal{H}[g]\in \mathcal{H}[G]$ satisfies
\[
\mathcal{H}[g](B^{H}_{R}(\overline{B}_{r}(x))) \cap B^{H}_{R}(\overline{B}_{r}(x)) \neq \varnothing.
\]
On the other hand
\[
\begin{array}{rcl}
\mathcal{H}[g](B^{H}_{R}(\overline{B}_{r}(x))) \cap B^{H}_{R}(\overline{B}_{r}(x)) &=&  g(B^{H}_{R}(\overline{B}_{r}(x))) \cap B^{H}_{R}(\overline{B}_{r}(x)) \\
&=& B^{H}_{R}(g(\overline{B}_{r}(x))) \cap B^{H}_{R}\overline{B}_{r}(x)) \\
&=& B^{H}_{R}(\overline{B}_{r}(g(x))) \cap B^{H}_{R}(\overline{B}_{r}(x)),
\end{array}
\]
then there exists $\overline{B}_{s}(y) \in B^{H}_{R}(\overline{B}_{r}(g(x))) \cap B^{H}_{R}(\overline{B}_{r}(x))$ and therefore this ball satisfies
\[
d_H(\overline{B}_{s}(y), \overline{B}_{r}(g(x))) < R \mbox{\ and \ } d_H(\overline{B}_{s}(y), \overline{B}_{r}(x)) < R.
\]
Further, as $(X,d)$ satisfies the shooting property we have the equalities
\[
\begin{array}{rcl}
d_H(\overline{B}_{s}(y), \overline{B}_{r}(g(x))) &=& d(y,g(x)) + |s-r|\\
d_H(\overline{B}_{s}(y), \overline{B}_{r}(x)) &=& d(y,x) + |s-r|,
\end{array}
\]
then $d(y,g(x))<R$ and $d(y,x)<R$. As a consequence of these two last inequalities we have $y\in g(B_{R}(x))\cap B_{R}(x)$, which implies $g\in H$. Hence
\[
\{ \mathcal{H}[g]\in\mathcal{H}[G] :  \mathcal{H}[g](B^{H}_{R}(\overline{B}_{r}(x))) \cap B^{H}_{R}(\overline{B}_{r}(x)) \neq \varnothing \} \subset \{ \mathcal{H}[g]: g\in H\},
\]
and thus the action $\mathcal{H}[G]\acts \Sigma(X)$ is proper.
\end{proof}

Using Proposition \ref{propershooting} we conclude ${\Sigma(X)}/{\mathcal{H}[G]}$ is a metric space where the metric is given by
\[
d_{\mathcal{H}[G]}(\mathcal{H}[G](A), \mathcal{H}[G](B)) = \dst\inf_{A,B\in \Sigma(X)} \{d_H(A,B)\},
\]
for $A,B\in \Sigma(X)$. Moreover, if $X$ is a length space, then ${X}/{G}$ and ${\Sigma(X)}/{\mathcal{H}[G]}$ are both length spaces.

\begin{lemma}
Let $(X,d)$ be a length space satisfying the shooting property and $G\acts X$. Then, for every $\overline{B}_{s}(y) \in \mathcal{H}[G](\overline{B}_{t}(x))$ we have $y\in G(x)$ and $s=t$.
\end{lemma}

\begin{proof}
Because $\overline{B}_{s}(y)$ belongs to the orbit of $\overline{B}_{t}(x)$ there exists $g\in G$ such that
\[
\overline{B}_{s}(y) = \mathcal{H}[g](\overline{B}_{t}(x)) = g(\overline{B}_{t}(x)) = \overline{B}_{t}(g(x)).
\]
Since $(X,d)$ satisfies the shooting property we get $s=t$ and $y=g(x)$ as we desired.
\end{proof}

\begin{theorem}\label{theorem:quotient}
Let $(X,d)$ be a length space satisfying the shooting property and $G\acts X$ a proper action. Then ${\Sigma(X)}/{\mathcal{H}[G]}$ is isometric to $({X}/{G})\times \mb{R}_{\geq 0}$ endowed with the taxicab metric $d_T$.
\end{theorem}

\begin{proof}
Fix $x_0,y_0\in X$ and $t_0,s_0\in \mb{R}_{\geq 0}$, thus
\[
\begin{array}{rcl}
d(\mathcal{H}(\overline{B}_{t_0}(x_0)) , \mathcal{H}(\overline{B}_{s_0}(y_0))) &=& \dst\inf\{d_H(\overline{B}_t(x),\overline{B}_s(y))\}  \\
&=& \dst\inf\{d_H(\overline{B}_t(x),\overline{B}_s(y))\}  \\
&=& \dst\inf\{ d(x,y) + |t-s| \}  \\
&=& \dst\inf\{d(g_x(x_0),g_y(y_0)) + |t_0-s_0|\} \\
&=& \dst\inf\{d(g_x(x_0),g_y(y_0))\} + |t_0-s_0| \\
&=& d_G(G(x),G(y)) + |t_0-s_0|,
\end{array}
\]
and the proof is complete.
\end{proof}

\begin{remark}
The presence of the shooting property allows us to transfer important properties from $X$ to $\Sigma(X)$ via proper isometric actions, but not every metric quotient space satisfies the shooting property even when the base satisfies it. For instance consider $\mb{R}$ with its standard metric and the group generated by the isometry $x\mapsto x+2\pi$. The quotient space is isometric to the circle $\mb{S}^{1}$ which does not satisfies the shooting property since the map $f$ from Theorem \ref{main} is not injective.
\end{remark}

\section*{Acknowledgements}

W. Barrera was partially supported by grants CONACYT-SNI 45382 and P/PFCE-2017-31MSU0098J-13. D. Solis was partially supported by grants CONACYT-SNI 38368, P/PFCE-2017-31MSU0098J-13 and UADY-CEA-SAB011-2017.  D. Solis want to acknowledge the kind hospitality of CIMAT- M\'erida, where part of this work was developed during a sabbatical leave.

%%%%%%%%%%%%%%%%%%%%
%%%%%%%%%%%%%%%%%%%%  References
%%%%%%%%%%%%%%%%%%%%

\vspace{1cm}

\noindent \textbf{Waldemar Barrera}. Facultad de Matem\'aticas, Universidad Aut\'onoma de Yucat\'an, Perif\'erico Norte Tablaje 13615,  C.P. 97110, M\'erida, M\'exico. \\
bvargas@correo.uady.mx

\vspace{.3cm}

\noindent \textbf{Didier A. Solis}.  ( \Letter \ ) Facultad de Matem\'aticas, Universidad Aut\'onoma de Yucat\'an, Perif\'erico Norte Tablaje 13615,  C.P. 97110, M\'erida, M\'exico. \\
didier.solis@correo.uady.mx

\vspace{.3cm}

\noindent \textbf{Luis M. Montes de Oca}. Facultad de Matem\'aticas, Universidad Aut\'onoma de Yucat\'an, Perif\'erico Norte Tablaje 13615,  C.P. 97110, M\'erida, M\'exico. \\
mauricio.montesdeoca@alumnos.uady.mx

\end{document}